\documentclass[11pt]{amsart}
\usepackage{amssymb,color,amsmath,amsfonts}
 \newtheorem{thm}{Theorem}[section]
 
 \newtheorem{lem}[thm]{Lemma}

 \theoremstyle{definition}

 \theoremstyle{remark}
 \numberwithin{equation}{subsection}
\newcommand{\inder}[2]{{\rm B^1}\left({#1},{#2}\right)}
\newcommand{\der}[2]{{\rm Der}({#1},{#2})}
\newcommand{\ho}[2]{{\rm Hom}({#1},{#2})}

\newcommand{\se}[1]{\big\{#1\big\}}

\newcommand{\cent}{{Z}(G)}
\newcommand{\frat}{{\rm \Phi}}

\newcommand{\om}[1]{\Omega_1({#1})}
\newcommand{\oms}[1]{\Omega_1^{\star}\left({#1}\right)}

\newcommand{\sder}[3]{{\rm Der}_{#1}({#2},{#3})}
\newcommand{\supp}[1]{{\rm supp}({#1})}
\newcommand{\di}[1]{{\rm{d}}\left({#1}\right)}

\begin{document}
\title[noninner automorphisms of order $p$]
{Finite $p$-groups of class 3 with
 noninner automorphisms of order $p$}
\author{Alireza Abdollahi}
\address{Department
 of Mathematics, University of Isfahan, Isfahan 81746-73441,
Iran; School of Mathematics, Institute for Research in Fundamental Sciences
(IPM), P.O.Box: 19395-5746, Tehran, Iran.}
\email{a.abdollahi@math.ui.ac.ir}
\author{Mohsen Ghoraishi}
\address{Department  of Mathematics, University of Isfahan, Isfahan 81746-73441,
Iran.}
\email{ghoraishi@gmail.com}
\thanks{This research was in part supported by a grant from IPM (No. 89200218)}
%\date{}
%\dedicatory{}
%%% ----------------------------------------------------------------------
\begin{abstract}
A longstanding conjecture asserts that every non-abelian finite
$p$-group $G$ admits a non-inner automorphism of order $p$.
The conjecture is valid for finite $p$-groups of class 2.
Here, we prove every finite non-abelian $p$-group $G$ of class $3$ with $p>2$
has a noninner automorphism of order $p$ leaving $\Phi(G)$ elementwise fixed. We also prove that if $G$ is a finite $2$-group of class $3$ which cannot be generated by $4$ elements,  then $G$ has a non-inner automorphism of order $2$ leaving $\Phi(G)$ elementwise fixed. We also prove that the  latter conclusion
 holds for finite $2$-groups $G$ of class $3$ such that the center of $G$ is not cyclic and the minimal number of generators of $G$ is  $2$ or $4$ and it holds whenever  the center of $G$ is {\em not} $2$-generated and the minimal number of generators of $G$ is $3$.
Some  results are also proved for the existence of non-inner automorphisms of order $p$ for a finite $p$-group $G$  under conditions in terms of the minimal number of generators of the center factor of $G$ and a certain function of the rank of $G$.
\end{abstract}
\maketitle
{\setlength{\baselineskip}{1.5\baselineskip}
\section{Introduction}
Let $G$ be a non-abelian finite  $p$-group. By a
celebrated result of Gasch\"{u}tz $G$ admits non-inner
automorphisms  of $p$-power order \cite{G}.  But the existence of a
non-inner automorphism of order $p$ for $G$  (Problem 4.13 of  \cite{Kbook})
has  faced  no counterexample yet.
Affirmative answers have been given to the problem
 whenever $G$ is nilpotent of class $2$ \cite{A,L}, $G$ is  regular
\cite{S,DS}, $G/Z(G)$ is powerful \cite{AB} and
 $C_G(Z(\frat(G)))\neq\frat(G)$ \cite{DS}.

Our main results imply that the title of our paper is valid.
The outline of the paper is as follows.

In Section 2, we look
 for  criteria to determine derivations from an
abelian $p$-group to an elementary abelian one,
 which yields  us a tool for finding a
lower bound for the size of certain subgroups of  group of derivations.

In Section 3, in
 order to find a non-inner automorphism of order $p$, we  compare the
order of certain subgroups of  derivations with the corresponding
subgroup of inner derivations.

Section 4 is devoted to prove the results indicated in the abstract.

Throughout the paper $p$ denotes a prime number.
 $\cent$, $G'$, $\frat(G)$ (or for short $\frat$), ${\rm cl}(G)$,
 ${\rm d}(G)$, and  ${\rm rk}(G)$
denote
 the center, the derived subgroup, the Frattini subgroup, the
nilpotency class and the minimum number of generators, respectively. If $G$
is a $p$-group, $\om{G}$ denotes
the subgroup generated by elements of order $p$ and $\oms G=\langle x\in
G|~x^p\in
Z(G)\rangle$. If $A$ is a  $G$-module,  $\der{G}{A}$ and $\inder{G}{A}$
 denote the groups of all derivations and inner derivations
 from $G$ to $A$, respectively.
If $a\in A$ and  $x\in G$ then $a^{1+x+\dots+x^{j}}$ denotes
$aa^x\cdots a^{x^{j}}$.
 The trace endomorphism induced by $x$ is denoted by
$\tau_x$ that sends $a\in A$ to
$a^{1+x+\dots +x^{n-1}}$ where $n=o(x)$.
\section{\bf An Estimation for the Minimum Number of Generators of $\der{G}{A}$}
We need the following lemma from \cite{C} to find a way to
determine all derivations from an abelian group $G$ to a $G$-module  $A$.
\begin{lem}[Lemma 1.2 of \cite{C}]\label{cut}
  Let $G=U\times V$ be a group, and let $A$ be a $G$-module. Let
$\alpha$ and $\beta$ be derivations from $U$ and $V$ respectively
to $A$. Then $\alpha$ and $\beta$ have a common extension to a
derivation from $G$ to $A$ if and only if $[u^\alpha, v]=[v^\beta,
u]$ for all $u \in U$ and $v \in V$. In particular, if $[U,A] = 1$
then $\der{G}{A}\cong\ho{U}{C_A(G)}\times \der{V}{A}$.
\end{lem}
The following lemma can be proved by using Lemma \ref{cut}.
\begin{lem}\label{der}
Suppose that $G=\langle x_1\rangle\times\dots\times\langle x_n\rangle$ is a
finite abelian group and
$A$ is a  $G$-module. Then the mapping
$\delta:{\big|}
\begin{smallmatrix}
x_i\rightarrow b_i\\
1\leq i\leq n
\end{smallmatrix}
$
can be extended to  an element of  
$\der{G}{A}$ if and only if
\begin{align*}
\left\{
\begin{matrix}
b_i\in\ker{\tau_{x_i}}~~~~~~~~~~~~~~~~~~~&1\leq i\leq n\\
[b_i,x_j]=[b_j,x_i] & 1\leq i,j\leq n
\end{matrix}
\right.
\end{align*}
\end{lem}
Suppose that  $G=\langle x_1\rangle\times\dots\times\langle x_n\rangle$ is an
abelian group and let
$A$ be a  $G$-module. If $\delta\in\der{G}{A}$, define
$\supp{\delta}$ to be $\se{{x_i}^\delta~|~1\leq i\leq n}$
 and for any non-empty subset $C$ of $A$, denote
$$\se{\delta\in\der{G}{A}|~ \supp{\delta}\subseteq C}$$
by $\sder{C}{G}{A}$. Recall that $\der{G}{A}$  is an abelian group with
pointwise multiplication of functions.
\begin{lem}\label{dimsder}
Let  $G=\langle x_1\rangle\times\cdots\times\langle x_n\rangle$ be a finite
abelian $p$-group and let $A$ be a finite
  elementary abelian  $p$-group  which is a  $G$-module.
 If  $C,D\leq A$ such that
 $\bigcup_{i=1}^n[C,x_i]\subseteq D$ then
\begin{align*}
{\rm d}(\sder{C}{G}{A})\geq \sum_{i=1}^n {\rm d}(\ker(\tau_{x_i})\cap
C)-\binom{n}{2}{\rm d}(D).
\end{align*}
\end{lem}
\begin{proof}
Let $C=\langle c_1, \dots,  c_r\rangle$ and
$D=\langle d_1,\dots, d_s\rangle$, where $\text{rk}(C)=r$ and $\text{rk}(D)=s$.
Then
$[c_j,x_i]= d_1^{m^1_{ij}}\cdots d_s^{m^s_{ij}}$ for some $m^k_{ij}\in
\mathbb{Z}_p$.
Thus for each $1\leq q\leq s$ we have a fixed $n\times r$ matrix
$[m_{ij}^q]\in \text{Mat}_{n\times r}(\mathbb{Z}_p)$.
Let $$\bold b=(b_1,\dots,b_n)\in(\ker(\tau_{x_1})\cap C)\times\cdots\times
(\ker(\tau_{x_n})\cap C).$$ Then for each $i$, there exist
$t_{1i},\dots,t_{ri}\in \mathbb{Z}_p$ such that
$$b_i=c_1^{t_{1i}}\cdots c_r^{t_{ri}}.$$
 It is easy to see that the map
$$M:(\ker(\tau_{x_1})\cap C)\times\cdots\times (\ker(\tau_{x_n})\cap C)
\to \mathbb{Z}_p^{rn}$$
 that sends $\bold b$ to
$M_\bold{b}=[t_{11},\dots,t_{r1},\dots,t_{1n},\dots,t_{rn}]$ is a group
monomorphism.

By Lemma \ref{der}, $\bold b$ determines a unique element of $\sder{C}{G}{A}$
 if and only if
\begin{gather}
\bold b\in (\ker(\tau_{x_1})\cap C)\times\cdots\times (\ker(\tau_{x_n})\cap C)
\;\;\text{and}\;
\tag{a}\\
\left\{
\begin{matrix}
\sum^r_{l=1}m^q_{il}t_{lj}=\sum^r_{l=1}m^q_{jl}t_{li}\\
 \text{for all} \; q\in\{1,\dots,s\},\; \text{for all}\; i,j\in\{1,\dots,n\}
\end{matrix}\right.\tag{b}
\end{gather}

It follows from the equalities (b) that
$[t_{11},\dots,t_{r1},\dots,t_{1n},\dots,t_{rn}]^{T}$
is a solution for the following matrix equation
\begin{align*}
{\left[
{\LARGE
\begin{smallmatrix}
& \vdots& &\vdots &
& \vdots& &\vdots&&\vdots&
\\
\\
0&\cdots& 0 & m^q_{j1}~~~~~~~~~\cdots ~~~~~~~~~ m^q_{jr}
 &0&\cdots&0 &-m^q_{i1}~~~~~~~~~~~~~~\cdots~~~~~~~~~ -m^q_{ir}&0&\cdots&0
\\
& \vdots & & \vdots&
& \vdots & &\vdots&&\vdots&
\end{smallmatrix}
}
\right]}_{_{\binom{n}{2}s\times rn}}
X=0
\end{align*}
Denote the latter matrix of coefficients by $E$ and consider the linear
transformation
 $$E_{C,D}:(\ker(\tau_{x_1})\cap C)\times\cdots\times (\ker(\tau_{x_n})\cap C)
\to \mathbb Z_p^{\binom{n}{2}s},$$ which sends $\bold b$ to
$E\times{M_\bold{b}}^{T}$. The  null space of $E_{C,D}$ is isomorphic
to $\sder{C}{G}{A}$. Therefore
\begin{align*}
{\rm d}(\sder{C}{G}{A})&=\dim({\rm null}(E_{C,D}))\\
&=\dim\big( (\ker(\tau_{x_1})\cap C)\times\cdots\times (\ker(\tau_{x_n})\cap C)
\big)
-\dim({\rm Im}(E_{C,D})),
\end{align*}
 This completes the proof.
 \end{proof}
 \section{\bf Non-inner automorphisms of order $p$}
Throughout this section, suppose that $G$ is a finite non-abelian  $p$-group
such that $C_G(Z(\frat(G)))=\frat(G)$. Note that the latter condition implies
$C_G(\Phi(G))=Z(\frat(G))$. We also fix throughout this section the
following notation:

$n:=\di{G}$,  $\overline{G}:=\frac{G}{\frat(G)}=\langle\bar{x}_1\rangle \times
\cdots\times \langle \bar{x}_n\rangle$,

$A:=\om{Z(\frat(G))}$,  $A_i:=A\cap Z_i(G)$ for all $i\in\mathbb{N}$, $A_i:=1$
for all negative integers $i$,

$A^\star:=\oms G\cap Z(\frat(G))$,

and note that
$\di{A_i}=\di{A^\star\cap
Z_i(G)}=\di{Z(\frat) \cap Z_i(G)} \geq {\rm d}(\frac{Z(\frat) \cap
Z_i(G)}{Z(G)})$, and
$A_1=\om{\cent}$ as $\cent\leq Z(\frat(G))=C_G(\Phi(G))$.

Consider  $A$ as a   $\overline{G}$-modules, where $\overline{G}$ acts by
conjugation on $A$.

\begin{lem}\label{prop}
If $x\in G$ then the trace endomorphism $\tau_x:A\to A$
has the following properties:
\begin{enumerate}
\item
${\rm d}(\ker(\tau_x))\geq \frac{1}{2} {\rm d}(A)$.
\item
$\tau_x(a)=[a,_{p-1}x]$ for all $a\in A$.
\item
If $0\leq i\leq p-1$ then $A_i\leq \ker(\tau_x)$.
\item
${\rm d}(\ker(\tau_x)\cap A_i)\geq
{\rm d}(A_i)-{\rm d}(A_{i-p+1})$.
\end{enumerate}
\end{lem}
\begin{proof}
(1) $\tau_x$ is a linear transformation and $\tau_x(a^x)=\tau_x(a)$, so
$\tau^2_x=0$. This proves (1).

(2) By induction on $m$ we have
$a^{x^m}=\Pi_{i=0}^m[a,_ix]^{\binom{m}{i}}$ where $[a,_0x]=a$. Now
\begin{align*}
\tau_x(a)
&=\Pi_{m=0}^{p-1}a^{x^m}\\
&=\Pi_{m=0}^{p-1}\Pi_{i=0}^m[a,_ix]^{\binom{m}{i}}\\
&=\Pi_{m=0}^{p-1}[a,_mx]^{\sum_{i=m}^{p-1}\binom{i}{m}}\\
&=\Pi_{m=0}^{p-1}[a,_mx]^{\binom{p}{m+1}}\\
&=\Pi_{m=0}^{p-1}[a^{\binom{p}{m+1}},_mx]\\
&=[a,_{p-1}x].
\end{align*}

(3) It  follows from (2).

(4) Note that $\tau_x$ maps $A_i$ into $A_i$. Thus
$${\rm d}(\ker(\tau_x)\cap A_i)={\rm d}(A_i)-{\rm d}(\text{Im}(\tau_x)\cap
A_i).$$
Now it follows from (2) that $\text{Im}(\tau_x)\cap A_i\subseteq A_{i-p+1}$.
This proves (4).
\end{proof}
Let $a\in Z(\frat(G))$. Then  the corresponding inner derivation induced by $a$
will be denoted by
$\varphi_a:\overline{G}\to Z(\frat)$
 which sends $\bar{x}$ to $[x,a]$.
\begin{lem}\label{theme}
If $\delta\in\der{\overline G}{A}$ then the map $\hat\delta:G\to G$ defined by
$x^{\hat\delta}=x{\bar x}^\delta$ for all $x\in G$, is an automorphism of order
$p$
 leaving $\frat(G)$  elementwise fixed. The automorphism $\hat\delta$ is  inner
if
and only if $\delta\in \inder{\overline G}{A^\star}$.
\end{lem}
\begin{proof}
It is straightforward.
\end{proof}
\begin{lem}\label{inner}
\begin{enumerate}
\item $\inder{\overline G}{A^\star}\bigcap \der{\overline
G}{A_{i-1}} =\inder{\overline G}{A^\star\cap Z_i(G)}$.
\item
$\inder{\overline G}{A^\star\cap Z_i(G)}\cong \frac{A^\star\cap
Z_i(G)}{\cent}$, \item $\di{\inder{\overline G}{A^\star\cap
Z_i(G)}}\leq {\rm d}(A_i)$.\\
 In particular,  if $\cent$ is a direct factor of
$A^\star$ then ${\rm d}(\inder{\overline G}{A^\star\cap
Z_i(G)})={\rm d}(A_i)-{\rm d}(A_1)$.
\item If $\di{\der{\overline
G}{A_{i-1}}}>{\rm d}(\inder{\overline G}{A^\star\cap Z_i(G)})$
then $G$ has a non-inner automorphism of order $p$ leaving $\Phi(G)$
elementwise fixed.
\item If $G$
has no non-inner automorphism of order $p$ leaving $\Phi(G)$ elementwise fixed,
then
$$\di{\der{\overline G}{A_{i-1}}}\leq \di{A_i}.$$
\end{enumerate}
\end{lem}
\begin{proof}
(1) It  is straightforward.

(2) It is enough to check  that the map
$$\varphi:A^\star\cap Z_i(G)\to
\inder{\overline G}{A^\star\cap Z_i(G)}$$
defined by $a^\varphi=\varphi_a$ for all $a\in A^\star\cap Z_i(G)$ is a group
epimorphism with the kernel $Z(G)$.

(3) It follows from (2).

(4) It follows from (2) and Lemma \ref{theme}.

(5) It follows from (3) and (4).
\end{proof}
\begin{thm}\label{general}
Suppose that   $\om{Z(\frat(G))}\leq Z_3(G)$.
If  $G$ satisfies one of the following conditions then $G$  has a non-inner
automorphism of order $p$ leaving $\Phi(G)$ elementwise fixed.
\begin{enumerate}
\item
$p>3$.
\item
$p=3$ and  ${\rm d}(G)> 3$.
\item
$p=3$, ${\rm d}(G)=3$, and
${\rm d}({Z(\frat(G))}\cap Z_2(G))\neq 3{\rm d}({Z(G)})$ or
${\rm d}({Z(\frat(G))})\neq 6{\rm d}(Z(G))$.
\item
$p=3$, ${\rm d}(G)=3$, and there exists $x\in G\backslash\frat$  such that
$\tau_x=0$.
\item
${\rm d}(G)=3$ and  $Z(G)$ is a direct factor of $\oms{G}\cap Z(\frat(G))$.
\end{enumerate}
\end{thm}
\begin{proof}
 Suppose, on the contrary, that $G$ has no non-inner automorphisms of order $p$
leaving $\Phi(G)$ elementwise fixed. In each cases we will arrive to a
contradiction.
 Lemma
\ref{dimsder} (by replacing $C$ and $D$ by $A_i$ and
$A_{i-1}$ respectively), implies  the  following inequalities,
for $i=1,2,3$
\begin{align*}
\di{\der{\overline{G}}{A_i}}
\geq\sum_{j=1}^n\di{\ker(\tau_{x_j})\cap A_i}-\binom{n}{2}\di{A_{i-1}},
\end{align*}
Setting $d_i=\di{A_i}$,
 it follows from
  Lemmas \ref{inner} and \ref{prop} that
%\begin{gather}\tag{F}
\begin{align}
%\begin{array}{l}
&d_2\geq{\rm d}(\frac{A^\star\cap Z_2(G)}{\cent})\geq nd_1,\tag{F1}\\
&d_3\geq {\rm d}(\frac{A^\star\cap Z_3(G)}{\cent})\geq
nd_2-nd_{3-p}-\binom{n}{2}d_1,\tag{F2}\\
&d_3\geq{\rm d}(\frac{A^\star}{\cent})
\geq\sum_{j=1}^3\di{\ker(\tau_{x_
j})}-\binom{n}{2}d_2
 \geq nd_3-nd_{4-p}-\binom{n}{2}d_2,\tag{F3}
%\end{array}
%\end{gather}
\end{align}
and therefore
%\begin{gather}
\begin{align}
%\begin{array}{l}
(n-1)d_{3}&\leq
 nd_{4-p}+\binom{n}{2}d_{2},\tag{F'1}\\
\binom{n}{2}d_{2}&\leq
nd_{4-p}+n(n-1)d_{3-p}+\binom{n}{2}(n-1)d_{1},\tag{F'2}\\
\binom{n}{2}d_{1}&\leq
nd_{4-p}+n(n-1)d_{3-p}\tag{F'3}
%\end{array}
%\end{gather}
\end{align}

(1) If $p>3$ then it follows from (F'3) that
$\binom{n}{2}d_1\leq 0$, which is a contradiction, since $d_1\geq 1$.

(2) It follows from (F'3) that
$\binom{n}{2}d_1\leq nd_1$. This proves (2).

(3) By hypothesis $n=3$. Since $\di{Z(\Phi(G))\cap Z_2(G)}=d_2$ and
$\di{Z(G)}=d_1$,  the above inequalities  imply that $6d_1=2d_2=d_3$.

(4) As $\di{G}=3$, by Burnside's basis theorem we may assume that $x_1=x$.
Since $\tau_{x_1}=0$, it follows from
inequality  (F3) that $d_3\geq 3d_3-2d_1-3d_2$ which is in contrary with (3).

(5) By
inequality (F2) and part (3) we have
$$
5d_1=d_3-d_1\geq{\rm d}(\frac{A^\star}{\cent}) \geq
3d_2-3d_1=6d_1,
$$
which is a contradiction.
\end{proof}

\begin{thm}\label{m1}
Suppose that $G$ is a finite $p$-group. If one the following
conditions hold then $G$ admits a non-inner automorphism of order
$p$ leaving $\Phi(G)$ elementwise fixed.
\begin{enumerate}
\item  $p$ is an odd prime and ${\rm
rk}(\frac{G}{Z(G)})<\binom{\di{G}+1}{2}\di{Z(G)}$, \item  $p=2$
and ${\rm rk}(\frac{G}{Z(G)})<\binom{\di{G}}{2}\di{Z(G)}$.
\end{enumerate}
\end{thm}
\begin{proof}
 Suppose, for a contradiction, that $G$ has no non-inner automorphisms of order $p$
leaving $\Phi(G)$ elementwise fixed. We use the notations and
inequalities (F1) and (F2) in the  proof of Theorem \ref{general}.
If $p$ is odd then
\begin{align}
& {\rm rk}(\frac{G}{Z(G)})\geq {\rm d}(\frac{A^\star\cap Z_3(G)}{\cent})\geq
nd_2-\binom{n}{2}d_1\geq
n^2d_1-\binom{n}{2}d_1=\binom{n+1}{2}d_1, \notag
\end{align}
and if $p=2$ then
\begin{align}
& {\rm rk}(\frac{G}{Z(G)})\geq {\rm d}(\frac{A^\star\cap Z_3(G)}{\cent})\geq
nd_2-nd_1-\binom{n}{2}d_1\geq
n^2d_1-nd_1-\binom{n}{2}d_1=\binom{n}{2}d_1. \notag
\end{align}
these imply the result.
\end{proof}

\section{\bf $p$-Groups of Class $3$}
Recall that the rank of a finite group $G$ is the minimum number $r$ such that
every subgroup of $G$ can be generated by at most $r$ elements. The number $r$
is denoted by ${\rm rk}(G)$.
\begin{lem}\label{rg}
If G is a nilpotent group of class $3$ then
${\rm rk}(\frac{G}{\cent})\leq\binom{{\rm d}(G)+1}{2}$.
\end{lem}
%%%%%%%%%%%%%%%%%%%%%%%%%%%%%%%%%%%%%%%%%%%%%%%%%%%%%%%%%%%%
\begin{proof}
Note that if $H$ is a  group of class $2$ then
${\rm rk}(H)\leq{\rm d}(\frac{H}{H'})+{\rm d}(H')$.
Suppose that
$G=\langle {x}_1,\dots,{x}_{{\rm d}(G)}\rangle$.
Then
$G'Z(G)=\langle [x_i,x_j], Z(G)| 1\leq i<j\leq {\rm d}(G)\rangle$. Therefore
$${\rm rk}(\frac{G}{Z(G)})\leq{\rm d}(\frac{G}{G'Z(G)})+
{\rm d}(\frac{G'Z(G)}{Z(G)})\leq{\rm d}(G)+\binom{{\rm d}(G)}{2}=\binom{{\rm
d}(G)+1}{2}.
$$
\end{proof}
\begin{lem}\label{cyclic}
Let $G$ be a finite non-abelian $p$-group of class $3$ where $p>2$.
 If $Z(G)$ is not cyclic then $G$ has a
non-inner automorphism of order $p$ leaving $\Phi(G)$ elementwise fixed.
\end{lem}
\begin{proof}
Suppose, for a contradiction, that $G$ does not satisfy the
conclusion of the lemma. Then by Theorem \ref{m1} and Lemma
\ref{rg} one has
$$\binom{\di{G}+1}{2}\di{Z(G)}\leq{\rm rk}(\frac{G}{\cent})\leq\binom{{\rm
d}(G)+1}{2}$$ that implies the result.
\end{proof}
We use the following lemma in the sequel without any further reference.
\begin{lem}\label{equ}
Let  $G$ be a $3$-group of class $3$,  $x,y,z\in G$ and $i\in\mathbb{N}$. Then
\begin{enumerate}
\item
$[x,y^i]=[x,y]^i[x,y,y]^{\binom{i}{2}}$,
\item
$[x^i,y]=[x,y]^i[x,y,x]^{\binom{i}{2}}$,
\item
$[x^3,y]=1$ if and only if $[x,y]^3=1$,
\item
$(xy)^3=x^3y^3[y,x]^3[y,x,x][y,x,y]^2$,
\item
$[x,y,z][y,z,x][z,x,y]=1$,
\item
$Z(\frat(G))\leq\oms G$.
\end{enumerate}
\end{lem}
\begin{thm}\label{c3}
For odd primes $p$, if
 $G$  is a finite non-abelian  $p$-group of class $3$
then $G$ has a non-inner automorphism of order $p$ leaving $\Phi(G)$ fixed
elementwise.
\end{thm}
\begin{proof}
Suppose, for a contradiction, that $G$ does not satisfy the conclusion of  the
theorem.
It follows from  \cite{DS} that $C_G(Z(\frat(G)))=\frat(G)$ and by Theorem
\ref{general}, we have $p=3$ and
$\di{G}\in\{2, 3\}$, and Lemma \ref{cyclic} implies that  $Z(G)$ is cyclic. Now
we
 distinguish  two cases: Case a: $\di{G}=3$; Case b: $\di{G}=2$.

\textbf{Case a.} Let $\di{G}=3$.

We will prove there exist $x_1,x_2,x_3\in G$, $w_1,w_2,w_3\in G'$ and  positive
integers $\alpha\geq\beta\geq\gamma$ such that
$$(\star) \;\;\;\;\ G=\langle x_1,x_2,x_3\rangle \;\;\text{and}\;\;
\om{Z(\frat(G))}=\langle  x_1^{3^\alpha}w_1, x_2^{3^\beta}w_2,
x_3^{3^\gamma}w_3,\om{G'}\rangle.$$

Suppose that $(\star)$ holds. Then  $[x,y,x_3^{3^{\gamma}}w_3]=1$ for all
$x,y\in\{x_1,x_2,x_3\}$.
Thus $[x,y,x_3]^{3^\gamma}=1$ for all
$x,y\in\{x_1,x_2,x_3\}$.
Now if $a\in \om{Z(\frat(G))}$ then
$a=x_1^{i3^\alpha}x_2^{j3^\beta}x_3^{k3^\gamma}w$ for some $w\in G'$ and
integers $i,j,k$. It follows from Lemma \ref{prop} that
$$\tau_{x_3}(a)=
[x_1^{i3^\alpha}x_2^{j3^\beta}x_3^{k3^\gamma}w,x_3,x_3]=
[x_1,x_3,x_3]^{i3^\alpha}[x_2,x_3,x_3]^{j3^\beta}=1,$$
and so $\tau_{x_3}=0$. This  latter contradicts part 4 of Theorem \ref{general}.

Hence, from now on we are going to show $(\star)$.
For this, we have to prove  several properties as follows. Suppose that
$G=\langle x_1,x_2,x_3 \rangle$.

1) \; $Z_2(G)\leq \frat=\Phi(G)$ and
 $\om{Z_2(G)}\leq Z(\frat)$.

For, by part 3 of Theorem \ref{general} we have
$\tau_x\neq 0$ for every $x\not \in \frat$. This implies that
$Z_2(G)\leq \frat$.
Since $a\in Z(\frat(G))$ if and only if
 $a\in \frat(G)$ and $[a,x^3]=1=[x,y,a]$ for all $x,y\in
G$,  the second assertion follows from Lemma
 \ref{equ}.

2) \; ${\rm d}(Z(G))=1$, ${\rm d}({Z(\frat(G))}\cap Z_2(G))=3$,
 ${\rm d}({Z(\frat(G))})=6$ and
 ${\rm d}(\frac{Z(\frat(G))}{Z(G)})=6$

These   follow from Lemma \ref{cyclic} and part 3 of
Theorem \ref{general}.

3) \; ${\rm d}(\frac{G'Z(G)}{Z(G)})=3$, $\di{G'}=3$ and $\gamma_3(G)\leq
\frat(G')=G'^3$.

For, by Lemma \ref{rg} and part (2), ${\rm rk}(\frac{G}{Z(G)})=6$. The
 proof of Lemma \ref{rg} shows that
${\rm d}(\frac{G'Z(G)}{Z(G)})=3$. Since
 $\gamma_3(G)\leq Z(G)$,  $$G'=\langle [x_1,x_2], [x_1,x_3], [x_2,x_3],
\gamma_3(G)\rangle.$$ Therefore
 $${\rm d}(\frac{G'Z(G)}{Z(G)})\leq {\rm d}(\om{G'})
\leq \di{\om{Z_2(G)}}=\di{\om{Z(\frat)}\cap Z_2(G)}=3.$$
 If $\gamma_3(G)\not\leq \frat(G')$ then
 Burnside's basis theorem implies that  ${\rm d}(\frac{G'Z(G)}{Z(G)})<3$,  a
contradiction.

4) \; If $x,y\in G$ then
$(xy)^3=x^3y^3c^3$, for some $c\in G'$.

5) \; If $x\in G$, $y\in G'$ then $(xy)^3=x^3y^3[y,x]^3$, and so
 $x^3y^3=w^3$ for some $w\in G$.

These parts follow from  Lemma \ref{equ} and part (3) above.

6) \; $G^3=\{x^3|x\in G\}$.

It can be easily verified.

7) \; $\om{\frat}=\{a\in \frat|a^3=1\}$.

For, if $a,b\in \frat$ and  $a^3=b^3=1$ then $a=a_1^3w_1$ for some $a_1\in G$
and $w_1\in G'$ and
so $[b,a,a]=[b,a_1^3w_1,a]=[b^3,a_1,a]=1$. Therefore  $(ab)^3=1$.

8) \; We may assume that $G=\langle x_1,x_2,x_3\rangle$ and
$$\frac{G^3}{G'^3}=\langle\overline{x_1^3}\rangle\times
\langle\overline{x_2^3}\rangle\times
\langle\overline{x_3^3}\rangle,$$
where $\overline{x}=xG'^3$ for any element $x\in G$.

For, first note that if $H=\langle a,b,c\rangle $ is a finite abelian $p$-group
such that $a$ has  the
 maximal  order in group $H$, and $b\langle
a\rangle$ has the  maximal order in $\frac{H}{\langle a\rangle}$ then
$$H=\langle a\rangle\times\langle a^ib\rangle\times\langle a^jb^kc\rangle$$
for some integers $i,j,k$. Thus  one may  assume that
$$\frac{G^3}{G'^3}=\langle\overline{x_1^3}\rangle\times
\langle\overline{x_1^{3i}}\overline{x_2^3}\rangle\times
\langle\overline{{x_1^{3j}}{x_3^{3k}}{x_3^3}\rangle},$$
for some integers $i,j,k$.
Now let $x'_1=x_1, x'_2=x_1^ix_2, x'_3=x_1^jx_2^kx_3$ then by Lemma \ref{equ}
and (4),
 $\{x'_1,x'_2,x'_3\}$ fulfil the conditions.

9)
 \; If $h\in G^3$, $w\in G'$ and $(hw)^3=1$ then $(hw)^3=h^3w^3$ and $hG'^3\in
\om{\frac{G^3}{G'^3}}$.

For, $1=(hw)^3=h^3w^3[w,h]^3=h^3w^3[w,hw]^3=h^3w^3$. Therefore $h^3\in G'^3$.

10) \;
It is clear that there are positive integers $\alpha,\beta,\gamma$ such that
$\om{\frac{G^3}{G'^3}}=\langle\overline{x_1^{3^\alpha}}\rangle\times
\langle\overline{x_2^{3^\beta}}\rangle\times
\langle\overline{x_3^{3^\gamma}}\rangle$. If $a\in \om{\frat(G)}$ then
$a=x_1^{i3^\alpha}x_2^{j3^\beta}x_3^{k3^\gamma}w$ for some $w\in G'$.

For, if $a=hw_0$ then by (9), $hG'^3\in \om{\frac{G^3}{G'^3}}$, thus
$h=x_1^{i3^\alpha}x_2^{j3^\beta}x_3^{k3^\gamma}w_1$ for some $w_1\in G'$.

11) \;
If $h_1,h_2\in G^3$ and $h_1^3, h_2^3\in G'^3$ then  $(h_1h_2)^3=h_1^3h_2^3$.

Since $h_1^3=c^3$ for some $c\in G'$,
$$[h_1^3,h_2]=[c^3,h_2]=[c,h_2^3]=1.$$ It follows that $[h_1,h_2]^3=1$,
 and since $h_2=k^3$ for some $k\in G$, we have $$[h_1,h_2,h_1]=
[h_1,k^3,h_1]=[h_1^3,k,h_1]=1.$$

12) \; If $a, b\in \om{\frat(G)}$ then $[a,b]\in \om{G'}$.

13) \; If $w, w_1 \in G'$, then  $w_1^3=w^3$ if and only if $w_1=ww_0$
 for some $w_0\in\om{G'}$.

Part (12) and (13) can be easily verified.

14) \; $\om{\frat(G)}=\langle  x_1^{3^\alpha}w_1, x_2^{3^\beta}w_2,
x_3^{3^\gamma}w_3,\om{G'}\rangle$, for some $w_1,w_2,w_3\in G'$ where
$\alpha,\beta,\gamma$ are as (10) and  $x_1^{3^{\alpha+1}}=w_1^{-3}$,
 $x_2^{3^{\beta+1}}=w_2^{-3}$ and $x_3^{3^{\gamma+1}}=w_3^{-3}$.

By (10) there are elements $w_1,w_2,w_3\in G'$ such that
$x_1^{3^{\alpha+1}}=w_1^{-3}$,
 $x_2^{3^{\beta+1}}=w_2^{-3}$ and $x_3^{3^{\gamma+1}}=w_3^{-3}$. Let  $a\in
\om{\frat(G)}$.  Then
$a= x_1^{i3^\alpha}x_2^{j3^\beta}x_3^{k3^\gamma}w$ for some $w\in G'$. Thus we
have
$$1\overset{(7)}{=}a^3\overset{(9)}{=}(x_1^{i3^\alpha}x_2^{j3^\beta}x_3^{
k3^\gamma})^3w^3
\overset{(11)}{=}x_1^{i3^{\alpha+1}}x_2^{j3^{\beta+1}}x_3^{k
3^{\gamma+1}}w^3=(w_
1^{-i}w_2^{-j}w_3^{-k}w)^3.$$ Therefore
$$w_1^{i}w_2^{j}w_3^{k}\equiv w \mod \om{G'}.$$
Now note that  $[b,u]\in \om{G'}$ whenever $b^3\in G'$ and $u\in G'$. This
implies that
$$a\equiv
(x_1^{3^\alpha}w_1)^i(x_2^{3^\beta}w_2)^j(x_3^{3^\gamma}w_3)^k \mod \om{G'}.$$

15) \; $\om{Z(\frat(G))}=\om{\frat(G)}$.

By (12),
 $\frac{\om{\frat(G)}}{\om{G'}}$  is an elementary abelian group of rank at
most $3$ and by (2) and (3) ${\rm rk}(\frac{\om{Z\frat(G)}}{\om{G'}})=3$.
Since
$\frac{\om{Z\frat(G)}}{\om{G'}}\leq \frac{\om{\frat(G)}}{\om{G'}}$, the result
follows.

This proves $(\star)$
and completes the proof in  the case (a).

\textbf{Case b.} Suppose that $\di{G}=2$.

Let $H=G/Z(G)$. We first prove that for any two elements $x,y\in G$, we have
$$H=\langle xZ(G),yZ(G)\rangle \Longrightarrow o(xZ(G))=o(yZ(G)) \;\;\;\; (**)$$

Suppose that $o(xZ(G))=3^{\alpha}$ and $o(yZ(G))=3^{\beta}$.
Suppose, for a contradiction, that $\alpha>\beta$. Then, it follows from Lemma
\ref{equ} that
$$[x^{3^\beta},y]=[x,y]^{3^\beta}[x,y,x]^{\frac{3^\beta(3^\beta-1)}{2}}=[x,y]^{3
^\beta}[x,y^{3^\beta},x]^{\frac{3^\beta-1}{2}}=[x,y]^{3^\beta},$$
since $y^{3^\beta}\in Z(G)$.
On the other hand,
$$1=[x,y^{3^\beta}]=[x,y]^{3^\beta}[x,y,y]^{\frac{3^\beta(3^\beta-1)}{2}}=[x,y]^
{3^\beta}[x,y^{3^\beta},y]^{\frac{3^\beta-1}{2}}=[x,y]^{3^\beta}.$$
Hence $[x^{3^\beta},y]=1$ and so $x^{3^\beta}\in Z(G)$, which is a
contradiction since we are assuming $o(xZ(G))=3^\alpha>3^\beta$.

Now let $x,y\in G$ be such that $H=\langle a,b\rangle$, where $a=xZ(G)$ and
$b=yZ(G)$. By $(**)$, $o(a)=o(b)=3^{\alpha}$ for some positive integer $\alpha$.

It is not hard to see (e.g., arguing as in the proof of  Lemma 2.4 of \cite{AB})
 that  $Z(H)=\langle a^{3^\beta},b^{3^\beta},[a,b]\rangle$, where
$o([a,b])=|H'|=3^\beta$ for some positive integer $\beta$.
Also we have $H/Z(H)$ and $H'=\langle [a,b]\rangle$ has the same exponent
$3^{\beta}$. By Corollary 2.3 of \cite{AB}, $d\big(Z(H))=2$ and so
$\alpha>\beta$.

 We first prove that $$\langle a^{3^\beta}\rangle \cap \langle
b^{3^\beta}\rangle=1. \;\;\;(***)$$ Let $a^{3^{\beta}i}=b^{3^{\beta}j}$ for
some $i,j\in \mathbb{Z}$.
 Assume that $i=3^{\beta_1}i'$ and $j=3^{\beta_2} j'$, where
$\gcd(3,i')=\gcd(3,j')=1$ and $\beta_1\geq \beta_2\geq 0$.
 Then $$(a^{3^{\beta_1-\beta_2}i'}b^{-j'})^{3^{\beta+\beta_2}}=a^{3^\beta
i}b^{-3^\beta j}
 [b^{-j'},a^{3^{\beta_1-\beta_2}i'}]^{\frac{3^{\beta+\beta_2}(3^{\beta+\beta_2}-
1)}{2}}=1.$$
 Since $H=\langle a, a^{3^{\beta_1-\beta_2}i'}b^{-j'}\rangle$,  it follows from
$(**)$ that $3^\alpha$ divides
 $3^{\beta+\beta_2}$ and so $a^{3^{\beta}i}=1=b^{3^{\beta}j}$.

Now we prove that either $\langle a^{3^\beta}\rangle \cap \langle
[a,b]\rangle=1$ or $\langle b^{3^\beta}\rangle \cap \langle [a,b]\rangle=1$.

Suppose, for a contradiction, that $a^{3^{\beta_1}}$ and $b^{3^{\beta_2}}$
belong to $\langle [a,b]\rangle$ for some non-negative integers
$\beta_1,\beta_2$ which are less than $\alpha$ and greater or equal to $\beta$.
Therefore $a^{3^{\beta_1}}=b^{3^{\beta_2}i}$ for some integer $i$ or
$a^{3^{\beta_1}j}=b^{3^{\beta_2}}$ for some integer $j$. It follows from
$(***)$ that either $\alpha\leq \beta_1$ or $\alpha\leq \beta_2$, a
contradiction.

By Corollary 2.3 of \cite{AB}, $d\big(Z(H))=2$. It follows that $a^{3^\beta
r}b^{3^\beta s}[a,b]^t=1$ for some integers $r,s,t$ such that
 $[a,b]^t\not=1$ and either $a^{3^\beta r}\not=1$ or $b^{3^\beta s}\not=1$. Let
$r=3^{\beta_1}r'$ and $s=3^{\beta_2}s'$, where $\gcd(r',3)=\gcd(s',3)=1$.
 Without loss of generality, assume that $\beta_2\leq \beta_1$. Then
$$(a^{3^{\beta_1-\beta_2}r'}b^{s'})^{3^{\beta+\beta_2}}\in \langle
[a,b]\rangle.$$

 Since  either $a^{3^\beta r}\not=1$ or $b^{3^\beta s}\not=1$,
$(a^{3^{\beta_1-\beta_2}r'}b^{s'})^{3^{\beta+\beta_2}}$ is a non-trivial
element of $H'$. Clearly $\{a, b'=a^{3^{\beta_1-\beta_2}r'}b^{s'}\}$ generates
$H$. Therefore, replacing $b$ by   $b'$, we obtain that
 $$\langle a^{3^\beta}\rangle \cap H'=1 \;\;\text{and}\;\;\ b^{3^{\alpha-1}}\in
H', \;\;\; (****)$$ since $1\not=b^{3^{\beta+\beta_2}}\in H'$.

Since $Z(G)$ is cyclic, $x^{3^{\alpha}}=y^{3^{\alpha}j}$ for some integer $j$
or  $x^{3^{\alpha}i}=y^{3^{\alpha}}$ for some integer $i$.
Without loss of generality, assume that $x^{3^{\alpha}}=y^{3^{\alpha}j}$. Since
$\alpha>\beta$,  $x^{3^{\alpha-1}}$ and $y^{-3^{\alpha-1}j}$are both in
$Z_2(G)$. Now we may write
$$(x^{3^{\alpha-1}}y^{-3^{\alpha-1}j})^3=x^{3^\alpha}y^{-3^{\alpha}j}
[y^{-3^{\alpha-1}j},x^{3^{\alpha-1}}]^{3}=[y^{-3^{\alpha-1}j},x^{3^{\alpha}}]=1.
$$
Let $k=x^{3^{\alpha-1}}y^{-3^{\alpha-1}j}$. Then $$k^3=1, k\in Z_2(G)\setminus
Z(G) \;\;\text{and}\;\; k\not\in G'. \;\;\;(\diamondsuit)$$ The latter holds:
for if $kZ(G)=a^{3^{\alpha-1}}b^{-3^{\alpha-1}j}\in H'$, then
$a^{3^{\alpha-1}}\in H'$, contradicting $(****)$.

Note that  $C_G(k)$ is a maximal subgroup of
$G$: for the map $x\mapsto [k,x]$ is a group homomorphism from $G$ onto the
cyclic group
$\Omega_1(Z(G))$ of  order $3$ with the kernel $C_G(k)$.
Let $x$ be any element of $G\setminus C_G(k)$.
Since $[k,x]\in Z(G)$ and $k^3=1$ we have
$$(xk)^{3}=x^{3}k^{3}[k,x]^{3}=x^3[k^3,x]=x^3.$$

It is easy to check that the map $\beta$ on $G$ defined by $(ux^i)^\beta=u
(xk)^i$ for all $u\in C_G(k)$
 and all integers $i$, defines an automorphism of order $3$ which leaves
$\Phi(G)$ elementwise fixed.
If $\beta$ were inner, then $\beta$ would be conjugation by some element $g\in
G\setminus Z_2(G)$ with $k=[x,g]$ and so $k\in G'$; contradicting
$(\diamondsuit)$. Thus $\beta$ is not inner. This completes the proof in the
case (b).
\end{proof}

\begin{thm}\label{Z<3}
Suppose that $G$ is a finite $2$-group of class $3$.
\begin{enumerate}
\item
 If $\di{G}\neq 3$ and $Z(G)$ is not cyclic
then $G$ has a non-inner automorphism of order $2$ leaving
$\Phi(G)$ elementwise fixed.
\item
 If $\di{G}=3$ and
$\di{Z(G)}>2$ then  $G$ has a non-inner automorphism of order $2$
leaving $\Phi(G)$ elementwise fixed.
\end{enumerate}
\end{thm}
\begin{proof}
Suppose to the contrary that $G$ has no non-inner automorphism of
order $2$. By part (2) of  Theorem \ref{m1} and Lemma \ref{rg} one
has
$$\binom{{\rm d}(G)+1}{2}\geq
{\rm rk}(\frac{G}{\cent})\geq\binom{{\rm d}(G)-1}{2}\di{Z(G)},$$
that implies the result when $\di{G}>3$. If $\di{G}=2, 3$, the
result follows by applying Lemma \ref{rg} and  inequality (F1) in
the proof of Theorem \ref{general}.
\end{proof}
\begin{thm}
Suppose that $G$ is a finite $2$-group of class $3$.
 If $\di{G}>4$ then
$G$ has a non-inner
automorphism of order $2$ leaving $\Phi(G)$ elementwise fixed.
\end{thm}
\begin{proof}
Suppose to the  contrary that $G$ has no non-inner automorphism of
order $2$. By
 Theorem \ref{Z<3},  $Z(G)$ is  cyclic  and  by  Lemma \ref{rg}
 and part (2) of Theorem \ref{m1}
$$
\binom{{\rm d}(G)}{2}\leq {\rm rk}(\frac{G}{Z(G)})\leq{\rm
d}(\frac{G}{G'Z(G)})+ {\rm d}(\frac{G'Z(G)}{Z(G)}),
$$
therefore
$${\rm d}(G')\geq{\rm d}(\frac{G'Z(G)}{Z(G)})\geq\binom{{\rm
d}(G)}{2}-\di{G}.$$ Since $\om{G'}\leq Z(\Phi(G))$  by applying
Lemma \ref{dimsder} for $C=\om{G'}$ and $D=\om{Z(G)}$ one has
\begin{align*}
\di{\der{\overline{G}}{\om{G'}}}
\geq\sum_{j=1}^{\di{G}}\di{\ker(\tau_{x_j})\cap
\om{G'}}-\binom{\di{G}}{2}\di{Z(G)}.
\end{align*}
and since $[\om{G'},x_j]\leq \om{Z(G)}$,   $1\leq i\leq \di{G}$
one has  $\di{\ker(\tau_{x_j})\cap \om{G'}}\geq \di{G'}-1$ and  by
Lemma \ref{rg}
\begin{align*}
{\rm rk}(\frac{G}{Z(G)})&\geq \di{\der{\overline{G}}{\om{G'}}} \geq \di{G}(\di{G'}-1)-\binom{\di{G}}{2}\Longrightarrow\\
\binom{{\rm d}(G)+1}{2}&\geq \di{G}(\binom{\di{G}}{2}-\di{G}-1)-\binom{\di{G}}{2}\Longrightarrow\\
0&\geq \di{G}^2-4\di{G}-2.
\end{align*}
therefore $\di{G}\leq 4$. This contradiction completes the proof.
\end{proof}

%%%%%%%%%--------------------------------------------------------

}

%%%%%%%%%%__________________________________________________

\begin{thebibliography}{99}
\bibitem{AB} A. Abdollahi, {\sl Powerful p-groups have noninner automorphisms
of order p and some cohomology},
J.  Algebra {\bf 323} (2010) 779-789.


\bibitem{A} A. Abdollahi, {\sl Finite $p$-groups of class 2 have noninner
automorphisms of order $p$}, J.  Algebra {\bf 312} (2007) 876-879.


\bibitem{C} G. Cutolo,
{\sl On a question about automorphisms of finite p-groups},
J.  Group Theory  {\bf 9} (2006) 231-250.

\bibitem{DS} M. Deaconescu, G. Silberberg, {\sl Noninner automorphisms of order
$p$ of finite $p$-groups}, J. Algebra {\bf 250} (2002) 283-287.

\bibitem{G} W. Gasch\"{u}tz, {\sl Nichtabelsche  $p$-Gruppen besitzen
\"{a}ussere $p$-Automorphismen}, J. Algebra {\bf 4} (1966) 1-2.

\bibitem{L} H. Liebeck, {\sl Outer automorphisms in nilpotent
 $p$-groups of class 2}, J. London Math. Soc. {\bf 40} (1965)
 268-275.

\bibitem{S} P. Schmid, {\sl A cohomological property of regular $p$-groups},
Math.
Z. {\bf 175} (1980) 1-3.

\bibitem{Kbook}
Unsolved problems in group theory, The Kourovka Notebook, No. 16,
Edited by V. D. Mazurov and E. I. Khukhro, Russian Academy of
Sciences, Siberian Division, Institue of Mathematics,
Novosibirisk, 2006.

\end{thebibliography}
\end{document}